\documentclass[12pt]{amsart}
\usepackage{amsthm, amssymb, stmaryrd, fullpage}
\usepackage{url}
\usepackage[all]{xy}

\DeclareMathAlphabet{\mathpzc}{OT1}{pzc}{m}{it}

\newtheorem{theorem}{Theorem}[subsection]
\numberwithin{equation}{theorem}
\newtheorem{lemma}[theorem]{Lemma}
\newtheorem{cor}[theorem]{Corollary}
\newtheorem{prop}[theorem]{Proposition}

\theoremstyle{definition}

\newtheorem{defn}[theorem]{Definition}
\newtheorem{example}[theorem]{Example}
\newtheorem{remark}[theorem]{Remark}
\newtheorem{hypothesis}[theorem]{Hypothesis}

\newcommand{\calC}{\mathcal{C}}
\newcommand{\calD}{\mathcal{D}}
\newcommand{\calE}{\mathcal{E}}
\newcommand{\calF}{\mathcal{F}}
\newcommand{\calI}{\mathcal{I}}
\newcommand{\calJ}{\mathcal{J}}
\newcommand{\calK}{\mathcal{K}}

\newcommand{\calO}{\mathcal{O}}
\newcommand{\calR}{\mathcal{R}}

\newcommand{\frakm}{\mathfrak{m}}
\newcommand{\frako}{\mathfrak{o}}
\newcommand{\frakp}{\mathfrak{p}}
\newcommand{\frakq}{\mathfrak{q}}

\newcommand{\CC}{\mathbb{C}}
\newcommand{\PP}{\mathbb{P}}
\newcommand{\QQ}{\mathbb{Q}}

\newcommand{\ZZ}{\mathbb{Z}}

\newcommand{\del}{\partial}

\newcommand{\be}{\mathbf{e}}

\DeclareMathOperator{\bDiv}{\mathbf{Div}}
\DeclareMathOperator{\bCDiv}{\mathbf{CDiv}}
\DeclareMathOperator{\CDiv}{CDiv}
\DeclareMathOperator{\Div}{Div}
\DeclareMathOperator{\divis}{divis}
\DeclareMathOperator{\End}{End}
\DeclareMathOperator{\Frac}{Frac}

\DeclareMathOperator{\Irr}{Irr}

\DeclareMathOperator{\rank}{rank}
\DeclareMathOperator{\red}{red}
\DeclareMathOperator{\RZ}{RZ}
\DeclareMathOperator{\Sch}{\mathbf{Sch}}
\DeclareMathOperator{\Spec}{Spec}

\begin{document}

\title{Good formal structures for flat meromorphic connections, III:
Irregularity and turning loci}
\author{Kiran S. Kedlaya}
\date{February 16, 2019}
\dedicatory{To Professor Masaki Kashiwara, on the occasion of his 70th birthday, with admiration}

\thanks{The author was supported by NSF (grants DMS-1101343, DMS-1501214, DMS-1802161), UC San Diego (Warschawski Professorship), a Guggenheim Fellowship (fall 2015), and IAS (Visiting Professorship 2018--2019), and additionally by NSF grant DMS-0932078
while visiting MSRI during fall 2014.}

\begin{abstract}
Given a formal flat meromorphic connection over 
an excellent scheme over a field of characteristic zero, 
in a previous paper we established existence of good formal structures and a good 
Deligne-Malgrange lattice after suitably blowing up.
In this paper, we reinterpret and refine these results by introducing some related structures.
We consider the \emph{turning locus}, which is the set of points
at which one cannot achieve a good formal structure without blowing up.
We show that when the polar divisor has normal crossings, the turning locus is of pure codimension 1 within the polar divisor, and hence of pure codimension 2 within the full space;
this had been previously established by Andr\'e in the case of a smooth polar divisor.
We also construct an \emph{irregularity sheaf} and its associated \emph{b-divisor}, which 
measure irregularity along divisors on blowups of the original space; this generalizes another result of Andr\'e
on the semicontinuity of irregularity in a curve fibration.
One concrete consequence of these refinements is a process for resolution of turning points which is functorial with respect to regular morphisms of excellent schemes; this allows us to transfer the result from schemes to formal schemes, complex analytic varieties, and nonarchimedean analytic varieties.
\end{abstract}

\maketitle

\section*{Introduction}

The Hukuhara-Levelt-Turrittin decomposition theorem
gives a classification of differential modules over the field
$\CC((z))$ of formal Laurent series resembling the decomposition of a
finite-dimensional vector space equipped with a linear endomorphism
into generalized eigenspaces. It implies that after adjoining
a suitable root of $z$, one can express any differential module as a successive
extension of one-dimensional modules.
This classification serves as the basis for
the asymptotic analysis of meromorphic connections around a 
(not necessarily regular) singular point.
In particular, it leads to a coherent description of the \emph{Stokes phenomenon},
i.e., the fact that the asymptotic growth of horizontal sections 
near a singularity must be described using different asymptotic series depending
on the direction along which one approaches the singularity.
(See \cite{varadarajan} for a beautiful exposition of this material.)

This is the third in a series of papers, starting with \cite{kedlaya-goodformal1, kedlaya-goodformal2}, in which we give some higher-dimensional analogues of the Hukuhara-Levelt-Turrittin
decomposition for \emph{irregular} flat formal meromorphic connections 
on complex analytic or algebraic varieties. 
(The regular case is already well 
understood by work of Deligne \cite{deligne}.)
Independently, similar results were obtained by Mochizuki \cite{mochizuki, mochizuki2}.
In the remainder of this introduction, we recall what was established in these prior papers,
explain what is added in this paper, and report some applications by other authors.

\subsection{Resolution of turning points}

In \cite{kedlaya-goodformal1},
we developed a 
numerical criterion for the existence of a \emph{good decomposition}
(in the sense of Malgrange \cite{malgrange-lille}) of a formal flat meromorphic
connection at a point where the polar divisor has normal crossings. 
This criterion is inspired by the treatment of
the original decomposition theorem given by Robba \cite{robba-hensel}
using spectral properties of differential operators on nonarchimedean rings;
our treatment depends heavily on joint work with
Xiao \cite{kedlaya-xiao} concerning differential modules on some nonarchimedean
analytic spaces.

We then applied this criterion to prove
a conjecture of Sabbah \cite[Conjecture~2.5.1]{sabbah}
concerning formal flat meromorphic connections on a two-dimensional
complex algebraic or analytic variety. 
We say that such a
connection has a \emph{good formal structure} at some point if it
acquires a good decomposition after pullback along a finite cover ramified only
over the polar divisor. In general, even if the polar divisor has normal 
crossings,
one only has good formal structures away from some discrete set, the set of
\emph{turning points} (hereafter called the \emph{turning locus}). However, Sabbah conjectured
that one can replace the given surface with a suitable blowup in such
a way that the pullback connection admits good formal structures everywhere;
we refer to such a blowup hereafter as a \emph{resolution of turning points}.
The construction uses the aforementioned numerical criterion plus some analysis on a 
certain space of valuations (called the \emph{valuative tree} by
Favre and Jonsson \cite{favre-jonsson}).

In \cite{kedlaya-goodformal2},
we constructed resolutions of turning points for formal flat meromorphic connections
on excellent schemes of characteristic zero, which include
algebraic varieties of all dimensions over any field of characteristic zero. This combined the numerical criterion of
\cite{kedlaya-goodformal1} with a more intricate valuation-theoretic argument,
based on the properties of one-dimensional Berkovich nonarchimedean
analytic spaces. 

We also obtained a partial result for complex analytic
varieties, using the fact that
the local ring of a complex analytic variety at a point is an
excellent ring. Namely, we obtained \emph{local} resolution of turning points, i.e., we only construct
a good modification in a neighborhood of a fixed starting point.
For excellent schemes, one can always extend the resulting local 
modifications, by taking the Zariski closure of the graph of a certain rational map,
then take a global modification dominating these.
However, this approach is not available for analytic varieties.

Independently, for flat meromorphic connections on projective varieties, resolutions of turning points were constructed\footnote{There is a minor technical discrepancy in the definition of good formal structures between our work and that of Mochizuki. The general existence of resolution of turning points is equivalent under both definitions, but some of our subsequent refinements do not carry over; see Remark~\ref{R:mochizuki def} for a detailed discussion.} by Mochizuki first in dimension 2 \cite{mochizuki} and then in general \cite{mochizuki2}. 
The approach is quite different, as the key argument uses positive-characteristic methods, particularly \emph{$p$-curvatures} in the sense of Katz \cite{katz1, katz2}. This follows in the vein of other positive-characteristic arguments in characteristic-zero algebraic geometry, such as Mori's bend-and-break lemma \cite[Chapter~3]{debarre}, in which one descends from a field to a finite-type $\QQ$-algebra and then reduces modulo a conveniently generic prime. Unfortunately, it is not clear how to extend such methods to the categories of formal schemes or complex-analytic varieties.

\subsection{Purity of the turning locus}

We start with a \emph{purity theorem} for the turning locus.
Let $X$ be a \emph{nondegenerate differential scheme} in the sense of \cite{kedlaya-goodformal2}; in particular, $X$ is a regular excellent $\QQ$-scheme (see Definition~\ref{D:nondegenerate} for the full definition). Given a meromorphic differential module on $X$ whose polar divisor $Z$ has normal crossings, we show that the turning locus is a closed subset of $Z$ of pure codimension 1; consequently, it has pure codimension 2 inside $X$.
As is typical for purity statements, such as Zariski--Nagata purity for branch loci
\cite[Tag 0BJE]{stacks-project}, this amounts to the statement that in the case where $X$ is the spectrum of a local ring of dimension at least 3, the turning locus cannot consist solely of the closed point.

Both the intuition and the proof of this statement rely on the fact that the turning locus can be interpreted in terms of Newton polygons. Loosely speaking, given an affine scheme $X = \Spec(R)$ and a meromorphic differential module over $X$, one can write down a monic univariate polynomial over $\Frac(R)$ whose Newton polygon computes the irregularity of the differential module along divisors of $X$. The poles of the coefficients of this polynomial constitute the polar divisor. If one restricts attention to those coefficients that define vertices of the Newton polygon, then the indeterminacy loci of these coefficients (i.e., the intersections of the zero loci with the polar divisor) constitute the turning locus.

One curious application of the purity theorem is the fact that one may perform resolution of turning points using a greedy algorithm, which alternates between resolution of singularities (to ensure that $X$ is regular and $Z$ has normal crossings) and blowing up in the (reduced) turning locus. (This does not depend on the choice of how to resolve singularities; we will comment further on this choice later.) Note that this argument does not itself give an independent proof of the existence of resolutions of turning points, as this is a key input into the proof.

\subsection{Irregularity b-divisors and sheaves}

We continue with a  repackaging of the results of \cite{kedlaya-goodformal1, kedlaya-goodformal2}, which helps shed some light on identifying resolutions of turning points among all modifications.
Again, let $\calE$ be a meromorphic differential module on a nondegenerate $\QQ$-scheme $X$.
Following Malgrange, we construct a corresponding \emph{irregularity function} on the set of exceptional
divisors on local modifications of $X$ (or equivalently, divisorial valuations on $X$).
One may view this function as a \emph{Weil divisor on the Riemann-Zariski space}
in the language of Boucksom-Favre-Jonsson \cite{boucksom-favre-jonsson},
or as a \emph{b-divisor} in the language of
Shokurov \cite{shokurov}; we adopt the latter terminology here.

The relationship between the irregularity function and resolutions of turning points can be summarized as follows. On one hand, there exists a certain Cartier divisor $D$ on a certain blowup $f: Y \to X$, called the \emph{irregularity b-divisor}, such that the irregularity function is computed by multiplicities of $D$. That is, to measure irregularity along any given exceptional divisor, one may construct a blowup $Y'$ of $Y$ on which the given exceptional divisor appears, and then measure the multiplicity of the pullback of $D$ to $Y'$ along this divisor. 
On the other hand, a blowup $f: Y \to X$ with $Y$ regular is a resolution of turning points if and only if the irregularity b-divisors of both $f^* \calE$ and $\End(f^* \calE) = f^* \End(\calE)$
correspond to Cartier divisors on $Y$ itself, rather than a further blowup.

For this reason, it is desirable to control more closely the structure of the irregularity b-divisor of $\calE$. What we show here (Theorem~\ref{T:nef})
is that it is a \emph{nef} b-divisor: it has nonnegative degree on curves contracted by $f$. 
This implies that there exists an integrally closed ideal sheaf on $X$, the \emph{irregularity sheaf},
whose associated b-divisor is the irregularity divisor (and which is completely functorial).
In particular, the turning locus is the set of points where at least one of the irregularity sheaves of $\calE$ or $\End(\calE)$ is not locally principal; moreover, one may construct a resolution of turning points by principalizing these two ideal sheaves and then resolving singularities.  Again, this does not give an independent proof of the existence of resolution of turning points, as this is a key input into the proof; that said, we can easily imagine that there exists a direct construction of the irregularity sheaf which does avoid the intricate valuation-theoretic arguments of \cite{kedlaya-goodformal2}.

Continuing in the way of speculation, we also point out that the irregularity sheaf may be of use in describing 
\emph{logarithmic characteristic cycles} for algebraic $\mathcal{D}$-modules,
as described in the rank 1 case by Kato \cite{kato}.
We make no attempt in this direction here.

\subsection{Functorial resolution of singularities and turning points}

We conclude by exhibiting resolutions of turning points which satisfy \emph{functoriality} for regular morphisms on the base space.
Here the adjective \emph{regular} does not perform its colloquial function
of distinguishing true morphisms of schemes from \emph{rational morphisms}, which
are only defined on a Zariski open dense subspace of the domain. Rather, a morphism of schemes is \emph{regular} if it is flat with geometrically regular
fibres; for instance, any smooth morphism is regular.
Even more specifically, open immersions and \'etale morphisms are regular, so functoriality for regular morphisms
implies locality for the Zariski and \'etale topologies.
This formalism is modeled on the formalism of functorial\footnote{This is sometimes called \emph{canonical} resolution
of singularities, but this misleadingly suggests a lack of arbitrary choices in the process.
Temkin's proofs can in principle be adapted to other functorial resolution algorithms for complex
algebraic varieties (several of which are described in \cite{hauser}); this should lead to different
(but still functorial) resolutions of singularities
for quasiexcellent schemes over a field of characteristic zero.}
(nonembedded and embedded) resolution of singularities for quasiexcellent schemes over a field of characteristic zero,
as established by Temkin \cite{temkin1, temkin2} using the resolution algorithm for complex algebraic
varieties given by Bierstone and Milman \cite{bierstone-milman, bmt}.

In the preceding discussions, we described two constructions of resolutions of turning points which combine resolutions of singularities with certain
modifications which depend on the specified differential module, in a manner that is functorial for regular morphisms
(either blowing up in the turning locus or principalizing irregularity sheaves).
By insisting upon Temkin's approach to resolution of singularities, we obtain resolutions of turning points which are themselves functorial for regular morphisms.

As with resolution of singularities, making resolution of turning points functorial for regular morphisms has the benefit that it allows the result to be globally transferred from schemes to other categories of interest, such as 
formal schemes, complex analytic varieties (or formal completions thereof), rigid analytic spaces, or Berkovich analytic spaces (Theorem~\ref{T:global}). In each of these categories, every object is covered by neighborhoods which are associated to a certain excellent ring (in the case of complex analytic varieties, one takes the stalk at a closed polydisc); for any inclusion of such neighborhoods, the associated transition map of rings induces isomorphisms of formal completions of closed points, and hence gives rise to a regular morphism of schemes.
For each such neighborhood, we may pass from the excellent ring to its associated scheme and apply resolution of turning points there; functoriality for regular morphism then allows for glueing of the resulting modifications in the desired category.

\subsection{Some related work}

To conclude this introduction, we survey some interactions between resolution of turning points and work
of various other authors.

\begin{itemize}

\item
Prior to any of our work on this topic, 
Andr\'e \cite{andre} studied the variation of irregularity in a complex-analytic family of meromorphic connections
and proved two results which are extended by our present work. One is essentially our purity theorem for the turning locus, but in the case of a smooth polar divisor (see Theorem~\ref{T:purity of the turning locus}). The other is a semicontinuity property for irregularity, which can be recovered from the construction of the irregularity sheaf (see Corollary~\ref{C:semicontinuity}).

This paper, together with \cite{kedlaya-goodformal1, kedlaya-goodformal2}, owe more of a debt to \cite{andre} than might be apparent. As remarked upon briefly at the end of the introduction to \cite{kedlaya-goodformal1}, it was a conversation in the wake of \cite{andre} in which Andr\'e originally suggested to us to attack Sabbah's conjecture
by transposing ideas from our work on semistable reduction for overconvergent $F$-isocrystals
\cite{kedlaya-semi1, kedlaya-semi2, kedlaya-semi3, kedlaya-semi4}.

\item
Asymptotic analysis and the Stokes
phenomenon have been treated in the two-dimensional case by Sabbah
\cite{sabbah}  (building on work of Majima \cite{majima}), conditioned on resolution of turning points.
The higher-dimensional case works similarly; see for example \cite{mochizuki2} or \cite{teyssier16}.

\item
Using resolution of turning points, D'Agnolo and Kashiwara \cite{dagnolo-kashiwara} have described a 
Riemann--Hilbert correspondence for $\mathcal{D}$-modules which are holonomic but not necessarily regular.

\item
An alternate characterization of the turning locus of $\calE$ has been given by Teyssier \cite{teyssier}: 
it is the locus on
the polar divisor where the \emph{solution complexes} of $\calE$ and $\End(\calE)$ are local systems.
This  provides a link between the irregularity b-divisor
and the \emph{irregularity complex} of Mebkhout \cite{mebkhout} which it would be useful to further clarify.
As emphasized
in the introduction of \cite{teyssier}, Teyssier's criterion contrasts with the numerical criteria used herein, by virtue of being a transcendental condition rather than an algebraic one. 

\end{itemize}

\subsection*{Acknowledgments}
Thanks to Sebastian Boucksom, Tommaso de Fernex, Chris Dodd, Mihai Fulger, Mattias Jonsson, 
Rob Lazarsfeld, James McKernan, Matthew Morrow, Mircea Musta\c{t}\u{a}, and Michael Temkin for helpful discussions.
Thanks also to the anonymous referee for detailed feedback, particularly on \S 2.2.

\section{Birational geometry of excellent schemes}
\label{sec:birational geometry}

We begin with some statements and results concerning the birational geometry of excellent $\QQ$-schemes using Shokurov's language of b-divisors.

\setcounter{theorem}{0}
\begin{hypothesis} \label{H:only schemes}
Throughout this paper, the only schemes we consider are noetherian, separated, excellent $\QQ$-schemes.
The only divisors we consider are integral (not rational or real) Weil and Cartier divisors.
\end{hypothesis}

\subsection{Riemann-Zariski spaces}

We start by recalling the definition of the Riemann-Zariski space associated to a scheme.

\begin{defn}
By a \emph{schematic pair}, we will mean a pair $(X,Z)$ in which $X$ is a scheme and $Z$ is a closed subscheme of $X$. We say such a pair is \emph{regular} (and describe it for short as a \emph{regular pair}) if $X$ is regular
and $Z$ is a normal crossings divisor on $X$.
By a \emph{morphism} $f: (X',Z') \to (X,Z)$ of schematic pairs, we will mean a morphism $f: X' \to X$ of schemes
for which $f^{-1}(Z) = Z'$; that is, for $\calI_Z$ the ideal sheaf defining $Z$,
the inverse image $f^{-1} \calI_Z \cdot \calO_{X'}$ should be the ideal sheaf defining $Z'$.
\end{defn}

\begin{defn}
By a \emph{modification} of schemes, we will mean a morphism $f: X' \to X$ which is proper, dominant, and an isomorphism away from a nowhere dense closed subset of $X$. The minimal such subset is called the \emph{center} of $f$.

By a \emph{regularizing modification} of a schematic pair $(X,Z)$, we will mean a 
morphism $(X',Z') \to (X,Z)$ of schematic pairs such that $X' \to X$ is a modification and $(X',Z')$ is a regular pair. We apply this definition to schemes by taking $Z = Z' = \emptyset$.
\end{defn}

We will use resolution of singularities for excellent schemes in the following form. We will give far more precise statements later (see Theorem~\ref{T:desing1} and Theorem~\ref{T:desing2}).
\begin{lemma} \label{L:desing1}
Let $(X,Z)$ be a schematic pair (as in Hypothesis~\ref{H:only schemes}) with $X$ reduced.
\begin{enumerate}
\item[(a)]
The pair $(X,Z)$ admits a regularizing modification $f: Y \to X$.
\item[(b)]
If $X$ is regular, then $f$ may be chosen to be a composition of blowups along regular centers.
\end{enumerate}
\end{lemma}

\begin{defn}
The \emph{Riemann-Zariski space} of a scheme $X$, denoted $\RZ(X)$, is the inverse limit
$\RZ(X) = \varprojlim Y$ of underlying topological spaces for $f: Y \to X$ running over all modifications of $X$
(or to avoid set-theoretic difficulties, a set of representatives of isomorphism classes of modifications).
For $x \in \RZ(X)$, write $x(Y)$ for the image of $x$ in $Y$.

We say a point $x \in \RZ(X)$ is \emph{divisorial} if for some modification $f: Y \to X$, $x(Y)$ is the generic point of some prime divisor of $Y$
(in which case we also say that $x$ is \emph{$f$-divisorial} or \emph{$Y$-divisorial}).
Let $\RZ^{\divis}(X)$ be the subset of $\RZ(X)$ consisting of divisorial points.
\end{defn}

\begin{remark}
Any dominant morphism $Y \to X$ of schemes induces a continuous map $\RZ(Y) \to \RZ(X)$
(compare Remark~\ref{R:dominant pullback} below); this is obviously a homeomorphism when $Y \to X$ is a modification.
In addition, for $X^{\red}$ the underlying reduced closed subscheme of $X$, the inclusion $X^{\red} \to X$ induces a homeomorphism
$\RZ(X^{\red}) \cong \RZ(X)$. There is thus little harm in assuming hereafter that $X$ is reduced.
\end{remark}

\begin{remark} \label{R:cofinal modifications}
One may equally well define $\RZ(X)$ using any cofinal set of modifications of $X$.
For example, by Lemma~\ref{L:desing1}, for $X$ reduced it suffices to consider regularizing modifications of $X$.
Also, since $X$ is excellent, the normalization of $X$ is a modification of $X$ consisting of a finite disjoint union of integral schemes $Y_i$, so $\RZ(X)$ is isomorphic to the disjoint union of the $\RZ(Y_i)$.
\end{remark}

\begin{remark} \label{R:Krull valuations}
For $X$ an integral scheme,
using the valuative criterion for properness,  we may identify $\RZ(X)$ with the set of equivalence classes of
Krull valuations $v$ on the function field $K(X)$ of $X$ such that
for some $x \in X$, the local ring $\calO_{X,x}$ is contained in the
valuation ring $\frako_v$ (we say that such valuations are \emph{centered on $X$}).
Under this identification, $\RZ^{\divis}(X)$ corresponds to the equivalence classes 
of \emph{divisorial valuations}, i.e., those valuations measuring
order of vanishing along some prime divisor on some modification of $X$.
\end{remark}

\subsection{The language of b-divisors}
\label{subsec:b-divisors}

We introduce the language of b-divisors (birational divisors) following \cite[\S 1]{boucksom-defernex-favre}, but with appropriate changes for the context of excellent $\QQ$-schemes rather than varieties over a field.
As noted above, we consider only integral divisors, rather than rational or real divisors.

\begin{hypothesis}
For the remainder of \S\ref{sec:birational geometry}, let $X$ be a reduced scheme
as in Hypothesis~\ref{H:only schemes}.
\end{hypothesis}

\begin{defn}
For $Y$ a reduced scheme, let $\Div Y$ and $\CDiv Y$ denote the groups of (integral) Weil and Cartier divisors, respectively, on $Y$. Recall that taking supports defines a natural morphism $\CDiv Y \to \Div Y$ which is an isomorphism when $Y$ is locally factorial \cite[Tag~0BE9]{stacks-project}, so in particular when $Y$ is regular. We do not need to consider rational or real Weil/Cartier divisors here.
\end{defn}

\begin{defn}
The group of \emph{(integral) b-divisors on $X$}, denoted
$\bDiv X$, is the group $\varprojlim_{Y \to X} \Div Y$, where $Y$ runs over modifications of $X$ and the transition maps are pushforwards.
For any modification $f: Y \to X$, the restriction
map $\bDiv X \to \bDiv Y$ is an isomorphism.
For $D \in \bDiv X$ and $f: Y \to X$ a modification, we refer to the component of $D$ in $\Div Y$ as the \emph{trace} of $D$ on $f$, and denote it by $D(Y)$.
\end{defn}

\begin{defn} \label{D:finiteness condition}
For $X$ the spectrum of a discrete valuation ring, we have natural identifications $\bDiv X = \Div X = \ZZ$.
For general $X$, this observation gives rise to a function from $\bDiv X$ to the set of integer-valued functions on
$\RZ^{\divis}(X)$: given $D \in \bDiv X$ and $x \in \RZ^{\divis}(X)$, choose a modification $f: Y \to X$ for which $x$ is $f$-divisorial, and then compute the image of $D$ in $\Div \Spec(\calO_{Y,x(Y)})$. This does not depend on the choice of $f$ because $\calO_{Y,x(Y)}$ does not depend on this choice.

Via this construction, we obtain an isomorphism of $\bDiv X$ with the group of functions $m: \RZ^{\divis}(X) \to \ZZ$
with the following finiteness property: for any modification $f: Y \to X$, there are only finitely many $f$-divisorial points $x \in \RZ^{\divis}(X)$ for which $m(x) \neq 0$. We use this interpretation to define the componentwise comparison relation $\leq$ on $\bDiv_X$.
\end{defn}

\begin{defn} \label{D:Cartier b-divisor}
The group of \emph{(integral) Cartier b-divisors on $X$}, denoted
$\bCDiv X$, is the group $\varinjlim_{Y \to X} \CDiv Y$, where $Y$ runs over modifications of $X$ and the transition maps are pullbacks. Again, for any modification $f: Y \to X$, the transition
map $\bCDiv X \to \bCDiv Y$ is an isomorphism.

The morphisms $\CDiv Y \to \Div Y$ induce a morphism $\bCDiv X \to \bDiv X$. Since we need only consider regularizing
modifications in light of Remark~\ref{R:cofinal modifications}, this morphism is injective; that is, we may interpret Cartier b-divisors as a special type of b-divisors.
\end{defn}

\begin{remark} \label{R:dominant pullback}
Let $X' \to X$ be a dominant morphism of reduced schemes. Then the pullback of any modification of $X$ is a modification of $X'$, so we obtain morphisms 
\[
\RZ(X') \to \RZ(X), \qquad \RZ^{\divis}(X') \to \RZ^{\divis}(X)
\]
and pullback morphisms 
\[
\bDiv X \to \bDiv X', \qquad \bCDiv X \to \bCDiv X'.
\]
\end{remark}

\begin{remark}
The term \emph{b-divisor} was introduced by Shokurov \cite{shokurov}
in his construction of 3-fold and 4-fold flips, but has since become
standard in birational geometry. See \cite{corti} for further discussion.
A very similar notion appears in the work of
Boucksom-Favre-Jonsson \cite{boucksom-favre-jonsson},
under the guise of Weil divisors on Riemann-Zariski spaces,
and is further developed in \cite{boucksom-defernex-favre}.
(In that language, Cartier b-divisors correspond to Cartier divisors on Riemann-Zariski spaces.)

The distinction between \emph{b-divisors} and \emph{Cartier b-divisors} has nothing to
do with the distinction between Weil and Cartier divisors on an individual space;
after all, the morphism from Cartier b-divisors to b-divisors uses the fact that Weil and Cartier divisors on a regularizing modification coincide (see Definition~\ref{D:Cartier b-divisor}). Rather, the terminology refers to the distinction between
pushforward functoriality for Weil divisors and pullback functoriality for Cartier divisors.
\end{remark}

\subsection{Determinations of Cartier b-divisors}

By definition, the data of a Cartier b-divisor does not include the specification of a particular Cartier divisor on a particular modification, and indeed there is not necessarily a preferred option. This leads us to the following definition.

\begin{defn}
For $D \in \bCDiv X$, a \emph{determination} of $D$ is a
modification $f: Y \to X$ such that $D$ belongs to the image of $\CDiv Y$ in $\bCDiv X$.
In this case, the trace $D(Y)$ is a Cartier divisor and is the element of $\CDiv Y$ mapping to $D$.
\end{defn}

\begin{remark}
Although we have opted not to do so here, it would be reasonable to refer to a determination of a Cartier b-divisor $D$  as a \emph{resolution} of $D$.
\end{remark}

\begin{lemma} \label{L:smooth centers}
For $X$ regular and $D \in \bCDiv X$, there exists a determination of $D$ which is a composition of blowups along regular centers.
\end{lemma}
\begin{proof}
Let $f: Y \to X$ be a determination of $D$. By choosing a relatively ample divisor for $f$ and pushing forward,
we can write $f$ as the blowup in some closed subscheme $Z$. We may apply Lemma~\ref{L:desing1} to the pair
$(X, Z)$ to conclude.
\end{proof}

\begin{defn}
For $D \in \bCDiv X$, the \emph{Cartier locus} of $D$ is the maximal dense subset $U$ of $X$ such that the restriction of $D$ to $\bCDiv U$ (in the sense of Remark~\ref{R:dominant pullback}) belongs also to $\CDiv U$. The complement of this set is the \emph{non-Cartier locus} of $D$. Note that for any determination $f: Y \to X$ of $D$, 
the non-Cartier locus is contained in the image of the support of $D(Y)$ in $X$,
and therefore is nowhere dense. Also, if $X$ is normal, then the non-Cartier locus has codimension at least 2 in $X$.

For $f: Y \to X$ a modification, we refer to the Cartier locus and the non-Cartier locus of $f^* D \in \bCDiv Y$ also as the \emph{Cartier locus} and \emph{non-Cartier locus} of $D$ on $Y$.
\end{defn}

\begin{defn}
If $f$ is a determination of $D \in \bCDiv X$, then the center of $f$ must contain the non-Cartier locus of $D$.
In general, it is not possible to choose a determination of $D$ with center equal to the non-Cartier locus of $D$; this is shown by the following example of Fulger taken from  \cite[Example~4.2]{boucksom-defernex-favre}.
\end{defn}

\begin{example} \label{exa:no determination}
Let $X$ be the affine 3-space over $\CC$ with origin $O$. 
Let $f_1: Y_1 \to X$ be the blowup along a line $L$ through $O$.
Let $f_2: Y \to Y_1$ be the blowup at a closed point $P$ in $f_1^{-1}(O)$.
Put $f = f_1 \circ f_2$. Then the exceptional divisor of $f_2$ may be viewed as an element
$D$ of $\bCDiv X$
with non-Cartier locus equal to $O$.

Suppose that $f': Y' \to X$ were a determination of $D$ with center $O$. The exceptional fibre $E$ of $f_1$ may be viewed as a $\PP^1$-bundle over $L$. Let $C_0$ (resp.\ $C_1$) be the strict transform in $Y$ of a section of $E \to L$ not passing through (resp.\ passing through) the point $P$. The intersection number $D(Y) \cdot C_i$
is then equal to $i$. On the other hand, if we write $L'$ for the strict transform of $L$ in $Y'$, then $D(Y) \cdot C_i = D(Y') \cdot L'$ for $i=0,1$, a contradiction.
\end{example}

\subsection{Relative nonnegativity}

We next introduce a relative nonnegativity property for Cartier b-divisors and show how it can be used to circumvent the issue appearing in Example~\ref{exa:no determination}.

\begin{defn}
Let $\calK_X$ denote the union of all fractional ideal sheaves on $X$.
By definition, an element of $\CDiv X$ is a global section of $\calK_X/\calO_X^\times$.

For $D \in \bCDiv X$, 
choose a determination $f: Y \to X$ of $D$ and let $\calI_X(D) \subseteq \calK_X$ be the union of all fractional ideal sheaves $\calJ$ on $X$ for which
$f^{-1} \calJ \cdot \calO_Y \subseteq \calO_Y(D(Y))$.
Note that $\calI_X(D)$ is itself a fractional ideal sheaf, and that it depends on $D$ but not on $f$.
\end{defn}

\begin{defn} \label{D:ideal to divisor}
Let $\calI$ be a fractional ideal sheaf on $X$. Using Definition~\ref{D:finiteness condition}, we may construct a b-divisor $D(\calI) \in \bDiv X$ whose corresponding function $\RZ^{\divis}(X) \to \ZZ$ takes $x$ to the multiplicity of $\calI$ along $x$. Let $f: Y \to X$ be the blowup of $X$ along $\calI$; by construction, $f^{-1} \calI \cdot \calO_Y$ is locally principal and so corresponds to an element of 
$\CDiv Y$. It follows that $D(\calI) \in \bCDiv X$.
Note that $\calI_X(D(\calI))$ is the integral closure of $\calI$, which in general need not equal $\calI$.
\end{defn}

\begin{defn} \label{D:basepoint-free}
For $D \in \bCDiv X$, we say that $D$ is \emph{basepoint-free} (relative to $X$) if 
$D(\calI_X(D)) = D$; equivalently, for some (hence any) determination $f: Y \to X$ of $D$, the adjunction map $f^* f_* \calO_Y(D(Y)) \to \calO_Y(D(Y))$
is an isomorphism. 
\end{defn}

\begin{remark}
Choose $D \in \bCDiv X$ and let $f: Y \to X$ be a determination of $D$.
Then by Nakayama's lemma, $D$ is basepoint-free if and only if for each $x \in X$,
the restriction of $\calO(D(Y))$ to $f^{-1}(x)^{\red}$ is generated by global sections.
By the theorem on formal functions, this means that the basepoint-free property may be checked by passing to the formal completion of each point of $X$.
\end{remark}

\begin{remark} \label{R:semiample determination}
If $D \in \bCDiv_X$ is basepoint-free, then (as in Definition~\ref{D:ideal to divisor}) blowing up $X$ along $\calI_X(D)$ yields a determination of $D$ with center equal to the non-Cartier locus of $D$; by the universal property of blowing up, this determination is the unique minimal determination of $D$. By contrast, for general $D \in \bCDiv_X$, there need not exist a unique minimal determination of $D$.
\end{remark}

\begin{defn}
For $D \in \bCDiv X$, we say that $D$ is \emph{nef} if there exists a determination $f: Y \to X$
of $D$ such that the pullback of $\calO_Y(D(Y))$ to each fibre of $f$ is nef; that is, for any commutative diagram
\begin{equation} \label{eq:nef diagram}
\xymatrix{
C \ar[r] \ar[d] & Y \ar^{f}[d] \\
\Spec(k) \ar[r] & X
}
\end{equation}
in which $k$ is an algebraically closed field and $C$ is a smooth proper connected curve over $k$, the pullback of $\calO_Y(D(Y))$ to $C$
has nonnegative degree. The same is then true for any other determination of $D$.
\end{defn}

\begin{lemma} \label{L:nef}
Suppose that $X$ is regular. Then for $D \in \bCDiv X$, $D$ is basepoint-free if and only if $D$ is nef.
\end{lemma}
\begin{proof}
Let $f: Y \to X$ be a determination of $D$.
If $D$ is basepoint-free, then for any commutative diagram as in \eqref{eq:nef diagram}, the pullback of $\calO_Y(D(Y))$  
is a line bundle which is globally generated, and hence has nonnegative degree.

To prove the converse, by Lemma~\ref{L:smooth centers} we may reduce to the case where
$f$ is the blowup along a smooth center $Z$. 
Then the only prime divisor of $Y$ which is not the proper transform of a prime divisor of $X$ is the exceptional divisor $E$ of $Y$. Let $m$ be the multiplicity of $E$ in $D$; then $D - mE \in \CDiv_X$, so $D$ is basepoint-free (resp.\ nef) if and only if $mE$ is. Since the degree of $E$ on every contracted curve is negative, $mE$ is nef if and only if $m<0$. On the other hand, $\calI_X(-E)$ is the ideal sheaf defining $Z$, so it is clear that $mE$ is basepoint-free for all $m<0$. This proves the claim in this case.
\end{proof}

\begin{remark}
In Lemma~\ref{L:nef}, it is probably critical that $X$ be regular. Otherwise, one can exert no control over the shape of the fibers of a determination (e.g., consider the affine cone over a projective variety over a field).
It should thus be possible to construct a counterexample against a generalization of Lemma~\ref{L:nef} using the fact that nef divisors on a smooth projective surface over an algebraically closed field need not be basepoint-free or even semiample (i.e., some positive integer multiple is basepoint-free). For a concrete example, see \cite[Example~11.12]{laface-testa}.
\end{remark}

In the spirit of Lemma~\ref{L:nef}, we state some related facts.
\begin{lemma} \label{L:condition for nef}
Suppose that $X$ is regular.
Choose $D \in \bCDiv_X$  and choose a determination of $D$ of the form
\[
Y = Y_n \stackrel{f_n}{\to} Y_{n-1} \cdots \stackrel{f_1}{\to} Y_0 = X
\]
where each $f_i$ is the blowup along a regular center
(this is always possible by Lemma~\ref{L:smooth centers}).
Suppose that for $i=1,\dots,n$, the inequality $D(Y_i) \leq D(Y_{i-1})$ holds in $\bCDiv_X$.
Then $D$ is nef.
\end{lemma}
\begin{proof}
Choose a diagram as in \eqref{eq:nef diagram}, and choose the smallest index $i$ for which
$C$ is not fully contracted in $Y_i$. By hypothesis, the difference $D(Y_i)-D(Y_{i-1})$ is a nonpositive multiple of the exceptional divisor of $f_i$; it therefore has nonnegative degree on $C$.
\end{proof}

The converse of this statement is the following.
\begin{lemma} \label{L:equality condition for determination}
Let $D \in \bCDiv_X$ be nef. Let $f: Y \to X$ be a regularizing modification
and view $D(Y) \in \CDiv_Y$ as an element of $\bCDiv_X$.
Then $D \leq D(Y)$, with equality if and only if $f$ is a determination of $D$.
\end{lemma}
\begin{proof}
By Lemma~\ref{L:smooth centers}, there exists a chain of blowups along regular centers
\[
Y_n \stackrel{f_n}{\to} Y_{n-1} \cdots \stackrel{f_1}{\to} Y_0 = Y
\]
which is a determination of $D$. Then for each $i$ we have an inequality $D(Y_i) \leq D(Y_{i-1})$ in $\bCDiv_X$,
with equality if and only if $D(Y_i) \in \bCDiv_{Y_{i-1}}$; this proves the claim.
\end{proof}

\begin{remark}
There is also a version of the nef condition for real b-divisors, using which one
can assert that a limit (for the locally convex direct limit topology) of nef b-divisors
is again nef. See \cite{boucksom-defernex-favre, boucksom-favre-jonsson}.
\end{remark}

\section{Purity of the turning locus}

In this section, we recall the basic properties of turning loci, and then establish a purity theorem for them (Theorem~\ref{T:purity of the turning locus}).

\subsection{Differential modules}

We review some basic definitions and terminology from \cite{kedlaya-goodformal1, kedlaya-goodformal2} concerning differential modules on schemes. We start by recalling \cite[Definition~3.1.2, Definition~3.2.2]{kedlaya-goodformal2}.

\begin{defn} \label{D:nondegenerate}
A \emph{nondegenerate differential scheme} is a pair $(X, \calD_X)$ in which $X$ is a scheme (as in Hypothesis~\ref{H:only schemes}), $\calD_X$ is a coherent sheaf equipped with an action on $\calO_X$ by derivations,
and for every point $x \in X$ there exist a regular sequence of parameters $x_1,\dots,x_n \in \calO_{X,x}$ and derivations $\del_1,\dots,\del_n \in \calD_{X,x}$ such that 
\[
\del_i(x_j) = \begin{cases} 1 & i=j \\ 0 & i \neq j. \end{cases}
\]
In particular, the scheme $X$ is regular.
See the appendix for an erratum to \cite{kedlaya-goodformal2} related to this definition.
\end{defn}

\begin{defn} \label{D:diff-module}
Let $(X,Z)$ be a schematic pair in which $X$ is equipped with the structure of a nondegenerate differential scheme
and $Z$ contains no component of $X$.
A \emph{$\nabla$-module} over $\calO_X(*Z)$ is a vector bundle $\calE$ on $X$ whose base extension to $\calO_X(*Z)$
is equipped with an action of $\calD_X$ satisfying the Leibniz rule.
\end{defn}

\begin{defn} \label{D:good formal structure}
With notation as in Definition~\ref{D:diff-module} and $x \in X$, 
a \emph{good formal structure} for $\calE$ at $x$ is a decomposition 
\[
\calE_{x} \otimes_{\calO_{X,x}} S \cong \bigoplus_{\alpha \in I}  E(\phi_\alpha) \otimes_S \calR_{\alpha}
\]
of $(\calD_{X,x} \otimes_{\calO_{X,x}} S)$-modules, where $S$ is a finite \'etale algebra over
$\widehat{\calO}_{X,x}(*Z)$ (the hat denoting completion along the intersection of the components of $Z$ containing $x$), $I$ is a finite index set,
$\calR_\alpha$ is a regular differential module over $S$,
and the $\phi_\alpha$ are elements of $S$ satisfying the following conditions. (Here $S_0$ denotes the integral closure of $\calO_{X,x}$ in $S$.)
\begin{enumerate}
\item[(i)]
For $\alpha \in I$, if $\phi_\alpha \notin S_0$, then $\phi_\alpha$ is a unit in $S$ and $\phi_\alpha^{-1} \in S_0$.
\item[(ii)]
For $\alpha,\beta \in I$, if $\phi_\alpha - \phi_\beta \notin S_0$, then $\phi_\alpha - \phi_\beta$ is a unit in $S$ and $(\phi_\alpha-\phi_\beta)^{-1} \in S_0$.
\end{enumerate}
\end{defn}

\begin{defn}
With notation as in Definition~\ref{D:diff-module}, a point $x \in X$ is a \emph{turning point} for $\calE$ if
$\calE$ does not admit a good formal structure at $x$. The set of turning points for $\calE$ is called the 
\emph{turning locus} of $\calE$. For $(X,Z)$ a regular pair, the turning locus is a nowhere dense closed subset of $Z$
\cite[Proposition~5.1.4]{kedlaya-goodformal2}. 

For $f: Y \to X$ a modification, the turning locus of $f^* \calE$ is contained in the inverse image of the turning locus of $\calE$. If the former is empty, we say that $f$ is a \emph{resolution of turning points} of $\calE$.
Note that $f$ need not be a regularizing modification of $(X,Z)$, but in practice we will usually add this condition when applying this definition.
\end{defn}

\begin{remark} \label{R:mochizuki def}
As noted in \cite[Remark~8.1.4]{kedlaya-goodformal2}, the definition of good formal structures here is essentially the one used by Sabbah \cite{sabbah}. Mochizuki \cite{mochizuki, mochizuki2} works with a more restrictive definition of good formal structures, so any turning point in our definition would be a turning point in Mochizuki's definition but not \emph{vice versa}. However, at any given point, if both $\calE$ and $\End(\calE)$ have good formal structures in our sense, then $\calE$ also has a good formal structure in Mochizuki's sense.
Consequently, for the totality of differential modules on a given class of schemes, the existences of resolutions of turning points in the two senses are equivalent; however, our subsequent interpretation of turning loci in terms of the irregularity b-divisor is only valid for good formal structures in the present sense.
\end{remark}

The main result of \cite{kedlaya-goodformal2} may then be stated as follows.
\begin{prop} \label{P:resolution of turning points}
With notation as in Definition~\ref{D:diff-module}, there exists a regularizing modification for $(X,Z)$ which is 
a resolution of turning points of $\calE$.
\end{prop}
\begin{proof}
See \cite[Theorem~8.1.3]{kedlaya-goodformal2}.
\end{proof}

\begin{remark}
The proof of Proposition~\ref{P:resolution of turning points} is an intricate valuation-theoretic calculation which gives very little control over the modification. The subsequent arguments in this paper give much more control of the modification, but as far as we know cannot be used to independently establish the existence of a resolution of turning points; they are thus largely complementary to the arguments of \cite{kedlaya-goodformal1, kedlaya-goodformal2}.
\end{remark}

\subsection{A local calculation}
\label{subsec:local purity}

We now make a local calculation, as described in the introduction, to obtain purity of the turning locus.
The description in the introduction refers to Newton polygons, but these appear only implicitly.

\begin{hypothesis} \label{H:dim 3 local comp}
Throughout \S\ref{subsec:local purity}, let $k$ be a field of characteristic $0$.
Fix an integer $n \geq 2$ and view $R := k \llbracket x_1,\dots,x_n \rrbracket$ as a nondegenerate differential ring via the derivations $\frac{\del}{\del x_1},\dots, \frac{\del}{\del x_n}$.
Put $X := \Spec(R)$, let $Z$ be the zero locus of $x_1\cdots x_n$, and let $z$ be the closed point of $X$.
Let $\calE$ be a $\nabla$-module of rank $d$ over $\calO_X(*Z)$ whose turning locus is contained in $\{z\}$.
\end{hypothesis}

\begin{defn} \label{D:parameter multiset}
For $x \in X$ not in the turning locus of $\calE$, set notation as in Definition~\ref{D:good formal structure},
then define the \emph{parameter multiset} of $\calE$ at $x$ to be the multisubset of $S/S_0$ obtained by including,
for each $\alpha \in I$,
the class of $\phi_\alpha$ with multiplicity $\rank \calR_\alpha$. This is independent of the choice of the decomposition, essentially because $E(\phi_\alpha - \phi_\beta)$ cannot be regular unless $\phi_\alpha - \phi_\beta \in S_0$.
\end{defn}

\begin{remark} \label{R:twist good formal}
Suppose that $\calE$ has empty turning locus; we may then take $x=z$ in Definition~\ref{D:good formal structure} and Definition~\ref{D:parameter multiset}. For any $\alpha \in R[x_1^{-1},\dots,x_n^{-1}]$ lifting an element of $S$,
one obtains a good formal structure for $\calE(-\alpha)$ at $x$ by twisting a good formal structure of $\calE$ at $x$;
consequently, $\calE(-\alpha)$ again has empty turning locus. By contrast, a more general twist does not
preserve good formal structures: any twist satisfies condition (ii) of Definition~\ref{D:good formal structure} but not necessarily condition (i).
\end{remark}

\begin{defn}
For $j=1,\dots,n$, let $\eta_j \in X$ be the generic point of the zero locus of $x_j$ in $X$.
Since $\eta_j$ is not in the turning locus, we may define $S_j$ to be the parameter multiset of $\calE$ at $\eta_j$.

For $j=1,\dots,n$, let $\eta'_j \in X$ be the generic point of the intersection of the components of $Z$ not containing $\eta_j$. Since $\eta'_j$ is also not in the turning locus, we may define $S'_j$ to be the parameter multiset of $\calE$ at $\eta'_j$. For $j' \neq j$, note that $S'_j$ projects to $S_{j'}$ by the well-definedness of the latter, although the exact matching of elements is in general not canonical due to the ring extension built into
Definition~\ref{D:good formal structure}.
\end{defn}

\begin{lemma} \label{L:parameter multiset}
For $\ell$ a finite extension of $k$ and $h$ a positive integer,
put
\[
R_{\ell,h} := \ell \llbracket x_1^{1/h},\dots,x_n^{1/h} \rrbracket.
\]
If $n \geq 3$, then there exists a unique multisubset $S$ of $\bigcup_{\ell,h} (R_{\ell,h}[x_1^{-1},\dots,x_n^{-1}]/R_{\ell,h})$
which projects to $S_j$ for $j=1,\dots,n$. (We refer to $S$ as the \emph{putative parameter multiset} of $\calE$ at $z$.)
\end{lemma}
\begin{proof}
Let $\overline{\alpha} \in S'_1$ be any element; it then lifts to an element
\[
\alpha \in \ell((x_1^{1/h}))\llbracket x_2^{1/h},\dots,x_n^{1/h} \rrbracket [x_2^{-1},\dots,x_n^{-1}]
\]
for some finite extension $\ell$ of $k$ and some positive integer $h$.
Choose any $j \in \{2,\dots,n\}$; since $n \geq 3$, we can choose an index $j' \in \{1,\dots,n\} \setminus \{1,j\}$ and equate both $S'_1$ and $S'_{j}$ with their projections to $S_{j'}$.
In so doing, we see that there exists a nonnegative rational number $e_1$ such that $x_1^{e_1} \alpha$
is integral over $R_{(x_1)}$; this implies that $\alpha \in R_{\ell,h}[x_1^{-1},\dots,x_n^{-1}]$.
By a similar argument, we see that we can choose $\ell,h$ so that $S'_1,\dots,S'_n$ are all multisubsets
of $R_{\ell,h}[x_1^{-1},\dots,x_n^{-1}]/R_{\ell,h}$.
Now the matching of the projections of $S'_1$ and $S'_j$ with $S_{j'}$ becomes canonical, so we obtain the desired result.
\end{proof}

We next introduce the numerical criterion for good formal structures
given in \cite[\S 4]{kedlaya-goodformal1}.
\begin{defn} \label{D:fi Fi}
For $r \in [0, +\infty)^n$, define the functions $f_i(\calE, r), F_i(\calE,r)$ for $i=1,\dots,d$
and $f_i(\End(\calE),r), F_i(\End(\calE),r)$ for $i=1,\dots,d^2$ as in \cite[Definition~3.2.1]{kedlaya-goodformal1};
in particular,  $f_i(*,\lambda r) = \lambda f_i(*,r)$ for all $\lambda \geq 0$ and $F_i(*,r) = f_1(*,r) + \cdots + f_i(*,r)$.
\end{defn}

\begin{lemma} \label{L:convex}
The functions $F_i(*, r)$ are continuous, piecewise linear, and convex.
\end{lemma}
\begin{proof}
Apply \cite[Theorem~3.2.2]{kedlaya-goodformal1}.
\end{proof}

\begin{lemma} \label{L:split off regular}
Suppose that there exists an index $i \in \{0,\dots,d-1\}$ such that:
\begin{enumerate}
\item[(i)]
$F_i(\calE,r)$ is linear in $r$;
\item[(ii)]
$f_{i+1}(\calE,r),\dots,f_d(\calE,r)$ are identically zero;
\item[(iii)]
either $i=0$, or $i>0$ and $f_i(\calE,r)$ is not identically zero.
\end{enumerate}
Then there exists a unique direct sum decomposition $\calE = \calE_1 \oplus \calE_2$
with $\rank(\calE_1) = i$ such that $\calE_2$ is regular.
\end{lemma}
\begin{proof}
Note that if $i>0$, then $F_{i-1}(\calE,r)$ is a convex function bounded above by the linear function $F_i(\calE,r)$; consequently, the two cannot agree at any point of $(0, +\infty)^n$. With this in mind, we may apply  \cite[Theorem~3.3.6]{kedlaya-goodformal1} to obtain the decomposition, and \cite[Theorem~4.1.4]{kedlaya-goodformal1} to see that $\calE_2$ is regular.
\end{proof}

\begin{remark} \label{R:turning locus decomposition}
We will use Lemma~\ref{L:split off regular} in conjunction with the following observation: if $\calE \cong \calE_1 \oplus \calE_2$ and $\calE_2$ is regular, then the turning loci of $\calE$ and $\calE_1$ coincide.
\end{remark}

\begin{lemma} \label{L:turning locus linearity}
The following statements are equivalent.
\begin{enumerate}
\item[(a)]
The turning locus of $\calE$ is empty.
\item[(b)]
The functions $F_1(\calE,r),\dots,F_d(\calE,r)$ and $F_{d^2}(\End(\calE),r)$ are linear in $r$.
\item[(c)]
The functions $F_d(\calE,r)$ and $F_{d^2}(\End(\calE),r)$ are linear in $r$.
\end{enumerate}
\end{lemma}
\begin{proof}
Apply \cite[Theorem~4.4.2]{kedlaya-goodformal1}. (See also Corollary~\ref{C:equality condition for determination} below.)
\end{proof}
\begin{remark} \label{R:End not linear}
In connection with Remark~\ref{R:mochizuki def}, we observe that if the turning locus of $\calE$ is empty, it
does not follow that $F_i(\End(\calE),r)$ is linear in $r$ for all of $i=1,\dots,d^2$;
this creates some complication in what follows.
See 
\cite[Example~4.4.5]{kedlaya-goodformal1} as well as 
Remark~\ref{R:End not linear} below.
\end{remark}

We may refine the statement of Lemma~\ref{L:turning locus linearity} as follows.
\begin{lemma} \label{L:admissible decomposition}
Suppose that $\calE$ has empty turning locus, and let $S$ be the parameter multiset of $\calE$ at $z$.
For $r \in [0, +\infty)^n$, let $v_r$ be the monomial valuation on $\bigcup_{\ell,h} R_{\ell,h}$ satisfying 
$v_r(x_i) = r_i$ for $i=1,\dots,n$. We then have equalities of multisets
\begin{align*}
\{f_i(\calE,r): i=1,\dots,d\} &= \{\max\{0,-v_r(\alpha)\}: \overline{\alpha} \in S\}, \\
\{f_i(\End(\calE),r): i=1,\dots,d^2\} &= \{\max\{0,-v_r(\alpha-\beta)\}: \overline{\alpha},\overline{\beta} \in S\},
\end{align*}
where $\alpha, \beta$ are lifts of $\overline{\alpha}, \overline{\beta}$.
\end{lemma}
\begin{proof}
Apply \cite[Lemma~2.5.3]{kedlaya-goodformal1} as in the proof of \cite[Theorem~4.4.2]{kedlaya-goodformal1}.
\end{proof}

\begin{lemma} \label{L:linear derivative}
Suppose that $n \geq 3$. 
For $j=1,\dots,n$, let $H_j$ be the set of $r = (r_1,\dots,r_n) \in [0, +\infty)^n$ for which $r_j = 0$. 
Then there exists an index $i_0 \in \{0,\dots,d\}$ for which the following statements hold.
\begin{enumerate}
\item[(a)]
The function $F_{i_0}(\calE,r)$ is linear in $r$ and identically equal to $F_d(\calE,r)$.
\item[(b)]
Either $i_0=0$, or $i_0 > 0$ and $f_{i_0}(\calE, r)|_{H_1 \cup \cdots \cup H_n}$ is not identically zero.
\end{enumerate}
\end{lemma}
\begin{proof}
Let $S$ be the putative parameter multiset of $\calE$ at $z$.
By Lemma~\ref{L:admissible decomposition},
we have an equality of multisets
\begin{equation} \label{eq:identify slopes1}
\{f_i(\calE,r): i=1,\dots,d\} = \{\max\{0,-v_{r}(\alpha)\}: \overline{\alpha} \in S\} \qquad (r \in H_1 \cup \cdots \cup H_n).
\end{equation}
For $j=1,\dots,n$, let $i_j \in \{0,\dots,m\}$ be the minimum index for which $F_{i_j}|_{H_j} = F_d|_{H_j}$ and 
put $i_0 := \max_j\{i_j\}$; then \eqref{eq:identify slopes1} yields
\begin{equation} \label{eq:identify slopes2}
F_{i_0}(\calE,r) = F_d(\calE,r) = -\sum_{\overline{\alpha} \in S} \max\{0,-v_{r}(\alpha)\} \qquad  (r \in H_1 \cup \cdots \cup H_n).
\end{equation}
If $i_0 = 0$, then Lemma~\ref{L:convex} implies that $F_d(\calE,r)$ is identically zero, which proves the claim in this case; we thus assume hereafter that $i_0 > 0$.
Choose $j$ for which $i_0 = i_j$; 
the restriction of $f_{i_0}(\calE,r)$ to $H_j$ is a linear (by \eqref{eq:identify slopes1}) function which is not identically zero,
so there exists another index $j' \neq j$ such that $f_{i_0}(\calE, \be_{j'}) \neq 0$.
By relabeling coordinates, we may reduce to the case where $j' = 1$; this forces $i_2 = \cdots = i_n = i_0$.

Let $T$ be the subset of $S$ consisting of elements which do not project to zero in $S_1$.
By Proposition~\ref{P:resolution of turning points} (see also Remark~\ref{R:eliminate resolution} below),
for $g$ a sufficiently large positive integer, the pullback of $\calE$
to $\Spec(k \llbracket x_1/(x_2 \cdots x_n)^g, x_2,\dots,x_n \rrbracket)$ has empty turning locus.
The parameter multiset of this pullback has the same projection to $S_1$ as $S$ does;
by applying Lemma~\ref{L:admissible decomposition} to the pullback, we obtain a neighborhood $U$ of $\be_1$ in $[0, +\infty)^n$ such that for $r \in U$,
\begin{equation} \label{eq:neighborhood fi}
\{f_i(\calE,r): i=1,\dots,i_0\} = \{-v_r(\alpha): \overline{\alpha} \in T\}.
\end{equation}
(Note that we cannot say anything here about $f_i(\calE,r)$ for $i > i_0$.)
For $s_2,\dots,s_n \geq 0$, by \eqref{eq:identify slopes2} we have
\begin{equation} \label{eq:identify slopes3}
\left. \frac{d}{dt} F_{i_0}(\calE, (1,ts_2,\dots,ts_n))\right|_{t=0^+}
= \sum_{\overline{\alpha} \in T} -v_{(0,s_2,\dots,s_n)}(\alpha) \\
= F_{i_0}(\calE, (0,s_2,\dots,s_n)) 
\end{equation}
when at least one of $s_2,\dots,s_n$ vanishes.
Since both sides of \eqref{eq:identify slopes3}
are linear in $s_2,\dots,s_n$ (the left by \eqref{eq:neighborhood fi}, the right by 
\eqref{eq:identify slopes1}),
\eqref{eq:identify slopes3} remains true for arbitrary $s_2,\dots,s_n \geq 0$.
By convexity (Lemma~\ref{L:convex}), we must have
\[
F_{i_0}(\calE, (1,ts_2,\dots,ts_n)) = F_{i_0}(\calE, \be_1) + t F_{i_0}(\calE, (0,s_2,\dots,s_n))
\]
for all $t \geq 0$; it follows that $F_{i_0}(\calE,r)$ is linear in $r$.
Meanwhile, we have the inequality
\[
F_d(\calE,r) \geq F_{i_0}(\calE,r)
\]
in which the left-hand side is convex, the right-hand side is linear, and equality holds for $r \in 
H_1 \cup \cdots \cup H_n$; we thus have equality for all $r$, proving the claim.
\end{proof}

\begin{remark} \label{R:End not linear}
Lemma~\ref{L:linear derivative} does not hold with $\calE$ replaced by $\End(\calE)$.
To see this, consider the example (with $n=3, d=6$) given by
\begin{gather*}
\calE = E(x_1^{-3} x_2^{-3} x_3^{-3})
\oplus
E(x_1^{-3} x_2^{-3} x_3^{-3} + x_1^{-1})
\oplus
E(x_1^{-2} x_2^{-2} x_3^{-2}) \\
\oplus
E(x_1^{-2} x_2^{-2} x_3^{-2} + x_2^{-1})
\oplus
E(x_1^{-1} x_2^{-1} x_3^{-1})
\oplus
E(x_1^{-1} x_2^{-1} x_3^{-1} + x_3^{-1});
\end{gather*}
in this case, $F_{d^2-10}(\calE,r) = F_{d^2}(\calE,r)$ and $f_{d^2-10}(\calE,r)$ is not identically zero, but does vanish for $r = \be_1,\be_2,\be_3$.
\end{remark}

\begin{lemma} \label{L:no turning locus codim 3}
Suppose that $n \geq 3$. Then the turning locus of $\calE$ is empty.
\end{lemma}
\begin{proof}
We proceed by induction on $d$. 
Let $S$ be the putative parameter multiset of $\calE$ at $z$.
Suppose first that $S$ contains only the zero element;
Lemma~\ref{L:admissible decomposition} then implies that $F_d(\calE,r) = 0$ for $r \in H_1 \cup \cdots \cup H_n$,
and Lemma~\ref{L:convex} then implies that $F_d(\calE,r)$ is identically zero.
By Lemma~\ref{L:split off regular}, $\calE$ is regular and we are done.

Suppose next that $S$ contains a nonzero element $\overline{\alpha}$. For the purposes
of checking the criterion of Lemma~\ref{L:turning locus linearity}, there is no harm in enlarging
$k$ or adjoining roots of $x_1,\dots,x_n$; we may thus assume without loss of generality that $\overline{\alpha}$ is the class of an element $\alpha \in R[x_1^{-1},\dots,x_n^{-1}]$. 
By Remark~\ref{R:twist good formal}, the turning locus of $\calE(-\alpha)$ is again contained in $\{z\}$,
and the putative parameter set of  $\calE(-\alpha)$  at $z$ is $S(-\alpha) := \{\overline{\beta}-\overline{\alpha}: \overline{\beta} \in S\}$.
By Lemma~\ref{L:linear derivative}, $F_d(\calE(-\alpha),r)$ is linear;
moreover, the value $i_0$ in Lemma~\ref{L:linear derivative} cannot equal $d$ because $0 \in S(-\alpha)$.
We may thus apply Lemma~\ref{L:split off regular} to split off a regular summand from $\calE(-\alpha)$,
then apply the induction hypothesis and Remark~\ref{R:turning locus decomposition} to deduce that 
the turning locus of $\calE(-\alpha)$ is empty. Since $\End(\calE) = \End(\calE(-\alpha))$,
Lemma~\ref{L:turning locus linearity} implies that $F_{d^2}(\End(\calE), r)$ is linear;
we may now apply Lemma~\ref{L:turning locus linearity}
again to deduce that the turning locus of $\calE$ is empty.
\end{proof}

\subsection{Globalization}

Our local calculation immediately globalizes to give a purity theorem for turning loci.
We recall from the introduction that in the case where $Z$ is smooth, the following result specializes to a theorem of Andr\'e \cite[Corollaire~3.4.3]{andre} modulo a change of categories (see Remark~\ref{R:transfer results}).

\begin{theorem} \label{T:purity of the turning locus}
Let $X$ be a nondegenerate differential scheme. Let $Z$ be a closed subscheme of $X$ such that $(X,Z)$ is a regular pair. Let $\calE$ be a $\nabla$-module over $\calO_X(*Z)$. Then the turning locus of $\calE$ on $X$ is a closed subscheme of $X$ of pure codimension $2$.
\end{theorem}
\begin{proof}
Suppose to the contrary that there exists an irreducible component of the turning locus of codimension $n \geq 3$.
Let $\eta$ be the generic point of this component, put $X' := \Spec(\calO_{X,\eta})$,
let $Z'$ be the pullback of $Z$ to $X$ as a Cartier divisor, and let $f: X' \to X$ be the canonical morphism.
Then $f^* \calE$ is a $\nabla$-module over $\calO_{X'}(*Z')$ whose turning locus consists of the closed point of $X'$. By taking formal completions and then applying Lemma~\ref{L:no turning locus codim 3} for the given value of $n$,
we deduce a contradiction.
\end{proof}

As observed in the introduction, this allows us to control the resolution of turning points
given by Proposition~\ref{P:resolution of turning points}.
\begin{cor} \label{C:resolution procedure1}
Let $X$ be a nondegenerate differential scheme. Let $Z$ be a closed subscheme of $X$ such that $(X,Z)$ is a regular pair. Let $\calE$ be a $\nabla$-module over $\calO_X(*Z)$. Define the modifications $f_n: X'_n \to X_n$, $g_n: X_{n+1} \to X_n$ as follows.
\begin{itemize}
\item
Set $X_0 = X$. Given $f_0,g_0,\dots,f_{n-1},g_{n-1}$ (resp.\ $f_0,g_0,\dots,f_{n-1},g_{n-1},f_n$),
let $Z_n$ (resp.\ $Z'_n$) be the inverse image of $Z$ in $X_n$ (resp.\ $X'_n$).
\item
Let $f_n$ be an embedded resolution of singularities of $(X_n, Z_n)$; that is, $f_n$ is a modification such that
$(X'_n, Z'_n)$ is a regular pair. (By Lemma~\ref{L:desing1}, such a modification always exists.)
\item
Let $g_n$ be the blowup of $X'_n$ in the reduced turning locus of the pullback of $\calE$ to $X'_n$.
\end{itemize}
Then for some $n_0$, the maps $g_n$ are isomorphisms for all $n \geq n_0$. For any $n \geq n_0$,
$X'_n \to X$ is a resolution of turning points of $\calE$.
\end{cor}
\begin{proof}
We may assume from the outset that $X$ is irreducible.
Suppose by way of contradiction that no such $n_0$ exists; this in particular means that for each $n$, the turning locus of the pullback of $\calE$ to $X'_n$ is a nonempty closed subset of $X'_n$, which we denote by $T_n$.
By Theorem~\ref{T:purity of the turning locus}, $T_n$ is of pure codimension 2 in $X'_n$.

By Proposition~\ref{P:resolution of turning points}, we may choose a resolution of turning points $h: Y \to X$ of $\calE$. Let $S$ be the finite set of divisorial valuations of $X$ corresponding to exceptional divisors of $h$.

For each $n$, let $h_n: Y_n \to X'_n$ be the proper transform of $h$ along $X'_n \to X$.
Let $\eta$ be the generic point of some component of $T_n$; then $h_n$ cannot be flat at $\eta$, or else
$\eta$ would not be a turning point. In particular, the inverse image of $\eta$ must be contained in some exceptional divisor of $h_n$, corresponding to some $v \in S$. Since $T_n$ is of pure codimension 2, the image of this exceptional divisor in $X'_n$ must be the closure of $\eta$ rather than some larger closed subspace.

Now note that any given $v \in S$ can occur only finitely many times in this fashion. This amounts to the following observation: for $X$ regular and excellent of dimension 2, any blowup can eventually be flattened by repeatedly blowing up in the reduced center. Since $S$ is itself finite, this yields the desired contradiction.
\end{proof}

\begin{remark} \label{R:eliminate resolution}
As written, the proof of Lemma~\ref{L:linear derivative} relies on resolution of turning points,
namely in the invocation of Proposition~\ref{P:resolution of turning points} to show that one can eliminate the turning locus by pulling back to $\Spec(k \llbracket x_1/(x_2 \cdots x_n)^g, x_2,\dots,x_n \rrbracket)$ for $g$ sufficiently large. However, it should be possible to give a more elementary proof of this by emulating the proof of \cite[Theorem~4.3.4]{kedlaya-goodformal2}. This in turn raises the possibility of using purity of the turning locus as the basis of a more elementary proof of resolution of turning points, in which one gives some other argument
(e.g., a finiteness argument based on considerations of cohomology) to establish the termination of 
the procedure described in Corollary~\ref{C:resolution procedure1}. We leave this as a question for future consideration.
\end{remark}

\section{Irregularity b-divisors}
\label{sec:irregularity}

In this section, we recast the main results of \cite{kedlaya-goodformal1, kedlaya-goodformal2} in the language of b-divisors, show that irregularity b-divisors are nef (Theorem~\ref{T:nef}), and use this to construct irregularity sheaves.

\setcounter{theorem}{0}
\begin{hypothesis}
Throughout \S\ref{sec:irregularity}, let $X$ be a nondegenerate differential scheme, let $Z$ be a closed subscheme of $X$ containing no component of $X$, and let $\calE$ be a $\nabla$-module of rank $d$ over $\calO_X(*Z)$.
\end{hypothesis}

\subsection{Irregularity b-divisors}
\label{subsec:irregularity}

\begin{defn}
By Definition~\ref{D:finiteness condition}, there exists a unique
b-divisor
$\Irr(\calE) \in \bDiv X$ such that for any modification $f: Y \to X$ with $Y$ normal and any prime divisor $E$ of $Y$,
the irregularity of $f^* \calE$ along $E$ equals the multiplicity of $\Irr(\calE)$ along $E$. 
We will see shortly that in fact $\Irr(\calE) \in \bCDiv X$ (Corollary~\ref{C:irregularity is Cartier});
we call $\Irr(\calE)$ the \emph{irregularity (Cartier) b-divisor} of $\calE$.
\end{defn}

\begin{prop} \label{P:numerical criterion}
Let $f: Y \to X$ be a regularizing modification of $(X,Z)$.
Then $f$ is a resolution of turning points of $\calE$ if and only if 
$\Irr(f^* \calE)$ and $\Irr(f^* \End(\calE))$ belong to the image of $\CDiv Y \to \bCDiv X\to \bDiv X$.
\end{prop}
\begin{proof}
This is immediate from \cite[Proposition~5.2.3]{kedlaya-goodformal2}.
\end{proof}

\begin{cor} \label{C:irregularity is Cartier}
The irregularity b-divisor $\Irr(\calE)$ is a Cartier b-divisor on $X$.
\end{cor}
\begin{proof}
This is immediate from Proposition~\ref{P:numerical criterion} plus the existence of a resolution of turning points
(Proposition~\ref{P:resolution of turning points}).
\end{proof}

Since $\Irr(\calE)$ is Cartier, we may formally restate Proposition~\ref{P:numerical criterion} as follows.
\begin{theorem} \label{T:numerical criterion}
For $f: Y \to X$ a regularizing modification of $(X,Z)$, the turning locus of $f^* \calE$ is the union of the non-Cartier loci of $\Irr(f^* \calE)$ and $\Irr(f^* \End(\calE))$.
Consequently, $f$ is a resolution of turning points if and only if $f$ is a determination
of both $\Irr(\calE)$ and $\Irr(\End(\calE))$. 
\end{theorem}
\begin{cor}
With no conditions on $Z$, the turning locus of $\calE$ is a closed subset of $X$ of codimension at least $2$.
(Recall that by Theorem~\ref{T:purity of the turning locus}, when $(X,Z)$ is a regular pair, the turning locus is of pure codimension $2$.)
\end{cor}
\begin{proof}
By Lemma~\ref{L:desing1}, there exists a regularizing modification $f: Y \to X$ for $(X,Z)$
which is an isomorphism in codimension 1.
Applying Theorem~\ref{T:numerical criterion} then yields the claim.
\end{proof}

\subsection{Irregularity sheaves}

We now check that the irregularity b-divisor is nef, and thus obtain the existence of irregularity sheaves.

\begin{theorem} \label{T:nef}
The Cartier b-divisor $\Irr(\calE)$ is nef.
\end{theorem}
\begin{proof}
By Lemma~\ref{L:condition for nef}, it suffices to check the following: for $f: Y \to X$ the blowup along a regular 
center $W$, the multiplicity $m$ of $\Irr(\calE)$ along the exceptional divisor of $f$ is nonpositive.
For this purpose, we may calculate at the generic point of $W$; that is, we  may assume that $X$ is the spectrum of a regular local ring and $W$ is its closed point. In this case, with notation as in Definition~\ref{D:fi Fi},
we have
\[
m = F_d(\calE, (1,\dots,1)) - F_d(\calE, (1,0,\dots,0)) - \cdots - F_d(\calE,(0,\dots,0,1));
\]
by the convexity of $F_d(\calE,r)$ as a function of $r$ \cite[Theorem~4.4.2]{kedlaya-goodformal1}, this quantity is nonpositive. This proves the claim.
\end{proof}

\begin{remark}
It is possible, but somewhat more complicated, to check directly that the degree of $\Irr(\calE)$ is nonnegative on any contracted curve. This argument would be analogous to an argument about $p$-adic connections made in \cite[Proposition~4.1.3]{kedlaya-swan2}; we refrain from including it here.
\end{remark}

Having just applied \cite[Theorem~4.4.2]{kedlaya-goodformal1}, we may now turn around and state a general result of which that statement is a special case.
\begin{cor} \label{C:equality condition for determination}
Let $f: Y \to X$ be a regularizing modification of $(X,Z)$. Let $D \in \CDiv_Y$ be the divisor supported on $f^{-1}(Z)$ in which the multiplicity of each prime divisor $E$ is the irregularity of $f^*\calE$ along $E$.
Then $\Irr(\calE) \leq D$, with equality if and only if $f$ is a determination of $\Irr(\calE)$.
\end{cor}
\begin{proof}
Combine Lemma~\ref{L:equality condition for determination} with Theorem~\ref{T:nef}.
\end{proof}

\begin{defn} \label{D:irregularity sheaf}
By Lemma~\ref{L:nef} and Theorem~\ref{T:nef}, the b-divisor $\Irr(\calE)$ is basepoint-free, and hence equals the b-divisor associated to the coherent ideal sheaf $\calI_X(\Irr(\calE))$. We refer to the latter as the \emph{irregularity sheaf} of $\calE$.
\end{defn}

Using irregularity sheaves, we obtain a second natural construction of resolutions of turning points.
\begin{cor} \label{C:resolution procedure2}
Let $f_1: X_1 \to X$ be the blowup of $X$ along the irregularity sheaf of $\calE$.
Let $f_2: X_2 \to X$ be the blowup of $X$ along the irregularity sheaf of $\Irr(\calE)$.
Let $g: Y \to X_1 \times_X X_2$ be a modification such that the composition $Y \to X$ is a regularizing modification of $(X,Z)$. Then $g$ is a resolution of turning points of $\calE$.
\end{cor}
\begin{proof}
This is immediate from Theorem~\ref{T:numerical criterion}.
\end{proof}

We also recover Andr\'e's semicontinuity theorem \cite[Corollaire~7.1.2]{andre}
modulo a change of categories (see Remark~\ref{R:transfer results}).
\begin{cor} \label{C:semicontinuity}
Suppose that $f: X \to S$ be a smooth morphism of nondegenerate differential schemes of relative dimension $1$
with connected fibers and that $Z$ is finite over $S$. Then the function assigning to a point
$x \in S$ the sum of the irregularities of $\calE|_{f^{-1}(x)}$ at all points $z \in Z \cap f^{-1}(x)$ 
is lower semicontinuous (in particular, it can only jump down under specialization).
\end{cor}
\begin{proof}
The function in question is locally constant away from the image in $S$ of the turning locus of $\calE$; it thus suffices to confirm the behavior under specialization. By pushing forward along a suitable finite morphism, we may further reduce to the case where $Z$ is a section of $f$. In this case, let $z \in Z$ be the unique preimage of $x \in S$ and put $C := f^{-1}(X)$. Let $g: Y \to X$ be a regularizing modification of $(X,Z)$ which is a resolution of turning points, and for which the proper transform $\tilde{C}$ of $C$ has transverse intersection with $g^{-1}(Z)$.
We may then compute the irregularity of $\calE|_C$ at $z$ on $\tilde{C}$ instead; the result is the multiplicity of the irregularity b-divisor of $\calE$ along the component of $g^{-1}(Z)$ meeting $\tilde{C}$. By Corollary~\ref{C:equality condition for determination}, we deduce the claim.
\end{proof}

\begin{remark}
Corollary~\ref{C:semicontinuity} does not require the full strength of the construction of irregularity sheaves.
It was previously observed by Sabbah \cite[Corollaire~3.2.4]{sabbah} that one can deduce the same assertion directly from the existence of resolution of turning points on surfaces.
\end{remark}

\section{Functoriality and change of categories}

We finally discuss functoriality for regular morphisms and describe functorial resolutions of turning points
in various geometric categories.

\subsection{Regular morphisms}

\begin{defn}
A morphism of rings $R \to S$ is \emph{regular} if it is flat and, for each prime ideal $\frakp$ of $R$, the fiber ring $S \otimes_R \kappa(\frakp)$ is noetherian and geometrically regular over $\kappa(\frakp)$. The geometrically regular condition means that for every finite extension $\ell$ of $\kappa(\frakp)$, $S \otimes_R \ell$ is regular;
this condition is only nontrivial for inseparable extensions \cite[Tag~038U]{stacks-project}, so for $\QQ$-algebras
it reduces to $S \otimes_R \kappa(\frakp)$  being regular.

A morphism $Y \to X$ of scheme is \emph{regular} if it is flat and, for each point $x \in X$,
the scheme $Y \times_X \Spec(\kappa(x))$ is locally noetherian and geometrically regular over $\kappa(x)$.
See \cite[\S 33]{matsumura-alg}, \cite[D\'efinition~6.8.2]{ega4-2}, or \cite[Tag~07R6]{stacks-project} for further discussion.
\end{defn}

\begin{lemma}\label{L:regular maximal}
For $R \to S$ a ring homomorphism, the following conditions are equivalent.
\begin{enumerate}
\item[(a)]
The morphism $R \to S$ is regular.
\item[(b)]
For each prime ideal $\frakq$ of $S$ lying over a prime ideal $\frakp$ of $R$, $R_\frakp \to S_\frakq$ is regular.
\item[(c)]
For each maximal ideal $\frakq$ of $S$ lying over a maximal ideal $\frakp$ of $R$, $R_\frakp \to S_\frakq$ is regular.
\end{enumerate}
\end{lemma}
\begin{proof}
See \cite[Tag~07C0]{stacks-project}.
\end{proof}

\begin{remark} \label{R:fppf}
For $X$ a scheme, the property of a morphism $f: Y \to X$ of schemes being regular is local with respect to the fppf topology on $X$ \cite[Tag~07RA]{stacks-project}. 
\end{remark}

\begin{remark} \label{R:smooth}
A morphism of schemes which is locally of finite presentation is smooth if and only if it is regular
(see \cite[Th\'eor\`eme~17.5.1]{ega4-4} or \cite[Tag~07R9]{stacks-project}).
In particular, a scheme locally of finite type over a field $K$ of characteristic 0 is smooth over $K$ if and only if it is regular.
\end{remark}

\begin{remark}
The definition of a regular morphism arises naturally as the relative version of regularity
for individual schemes. Unfortunately, it contravenes an entrenched convention in algebraic geometry, 
in which the term \emph{regular morphism} is used as an emphatic term for a \emph{morphism}, to contrast it with
a \emph{rational morphism} which is not really a morphism at all (being defined only on a dense open subset of the domain).
We will make no use of this convention.
\end{remark}

\subsection{Functorial resolution of singularities}

In \cite{kedlaya-goodformal2}, extensive use was made of the fact that
quasiexcellent $\QQ$-schemes admit nonembedded
and embedded desingularization; this was originally proposed by
Grothendieck, but only recently verified by Temkin \cite{temkin-excellent}.
In order to transfer resolution of turning points from schemes to other categories, we need
to perform resolutions of singularities in a manner which is functorial with respect to regular morphisms.
This can be achieved using approximation arguments, provided that one starts with a resolution algorithm for varieties over a field of characteristic 0 in which one repeatedly blows up so as to reduce some local invariant.
While this description does not apply to Hironaka's original proof of resolution of singularities,
it applies to several subsequent arguments, such as the method of Bierstone--Milman
\cite{bierstone-milman} as refined by Bierstone, Milman, and Temkin \cite{bmt}.
Using this method, Temkin has established the following
functorial desingularization theorems. (Temkin also obtains some control over the sequence of blowups used;
we have not attempted to exert such control in the following statements.)

\begin{defn}
Let $\Sch$ be the category of schemes.
Let $\Sch'$ be the category of schematic pairs.
\end{defn}

\begin{theorem}[Temkin] \label{T:desing1}
Let $\calC$ be the subcategory of $\Sch$ whose objects are the
reduced integral noetherian excellent $\QQ$-schemes,
and whose morphisms are the regular morphisms of schemes.
Let $\iota: \calC \to \Sch$ denote the inclusion.
There then exist a covariant functor $Y: \calC \to \Sch$ and a natural transformation $F: Y \to \iota$
satisfying the following conditions.
\begin{enumerate}
\item[(a)]
For each $X \in \calC$,  the morphism $F(X): Y(X) \to X$ is a projective regularizing modification.
\item[(b)]
For each regular $X \in \calC$, $F(X)$ is an isomorphism.
\item[(c)]
For each morphism $f: X' \to X$ in $\calC$, the square
\[
\xymatrix{
Y(X') \ar^{Y(f)}[r] \ar^{F(X')}[d] & Y(X) \ar^{F(X)}[d] \\
X' \ar^{f}[r] & X
}
\]
is cartesian in $\Sch$.
\end{enumerate}
\end{theorem}
\begin{proof}
See \cite[Theorem~1.2.1]{temkin1}.
\end{proof}

\begin{theorem}[Temkin] \label{T:desing2}
Let $\calC'$ be the subcategory of $\Sch'$ whose objects are the
pairs for which the underlying schemes are regular integral noetherian excellent $\QQ$-schemes,
and whose morphisms are those for which the underlying morphisms of schemes are regular.
Let $\iota': \calC' \to \Sch'$ denote the inclusion.
Then there exist a covariant functor $(Y,W): \calC' \to \Sch'$ and 
a natural transformation $F: (Y, W) \to \iota$ satisfying the following conditions.
\begin{enumerate}
\item[(a)]
For each $(X, Z) \in \calC'$, the morphism $F(X,Z): Y(X,Z) \to X$ is a projective regularizing modification of $(X,Z)$ (and even a sequence of blowups along regular centers).
\item[(b)]
For each regular $(X, Z) \in \calC'$, $F(X,Z)$ is an isomorphism.
\item[(c)]
For each morphism $f: (X',Z') \to (X,Z)$ in $\calC$, the square
\[
\xymatrix
@C=50pt{
(Y,W)(X',Z') \ar^{(Y,W)(f)}[r] \ar^{F(X',Z')}[d] & (Y,W)(X,Z) \ar^{F(X,Z)}[d] \\
(X',Z') \ar^{f}[r] & (X,Z)
}
\]
is cartesian in $\Sch'$.
\end{enumerate}
\end{theorem}
\begin{proof}
See \cite[Theorem~1.1.6]{temkin2}.
\end{proof}

\begin{remark}
A key special case of functoriality in both Theorem~\ref{T:desing1} and 
Theorem~\ref{T:desing2} is that of open immersions. This case implies that the modifications in question blow up in the smallest possible centers. Namely, in Theorem~\ref{T:desing1}, $F(X)$ is an isomorphism over the maximal regular open subscheme of $X$;
in Theorem~\ref{T:desing2}, $F(X,Z)$ is an isomorphism over the maximal open subscheme of $X$ on which $Z$ is a normal crossings divisor.
\end{remark}

\subsection{Functorial resolution of turning points}

We now state a theorem on the functorial resolution of turning points, which follows by combining
functorial resolution of singularities with our preceding arguments.

\begin{defn} \label{D:nabla-module category}
Let $\calC$ be the following category.
\begin{itemize}
\item
The objects of $\calC$ are tuples $(X,Z,\calE)$ in which $(X,Z)$ is a schematic pair,
$X$ is a nondegenerate differential scheme,
$Z$ contains no connected component of $X$,
and $\calE$ is a $\nabla$-module over $\calO_X(*Z)$.
\item
For two objects $(X,Z,\calE),(X',Z',\calE')$ of $\calC$, a morphism
$(X',Z',\calE') \to (X,Z,\calE)$ consists of a morphism $f:(X',Z') \to (X,Z)$ of schematic pairs
with $f: X' \to X$ regular,
a promotion of $f$ to a morphism of differential schemes,
and an isomorphism $\calE' \cong f^* \calE$ of $\nabla$-modules over $\calO_{X'}(*Z')$.
\end{itemize}
Let $\iota: \calC \to \Sch'$ denote the functor $(X,Z,\calE) \mapsto (X,Z)$.
\end{defn}

\begin{lemma} \label{L:regular diff}
Let $f: (X',Z',\calE') \to (X,Z,\calE)$ be a morphism in $\calC$.
\begin{enumerate}
\item[(a)]
Let $T$ (resp.\ $T'$) denote the turning locus of $\calE$ on $X$ (resp.\ the turning locus of $\calE'$ on $X'$).
Then $f^{-1}(T') = T$.
\item[(b)]
The ideal sheaf $f^*(\calI_X(\Irr(\calE))$ coincides with $\calI_{X'}(\Irr(\calE'))$.
\end{enumerate}
\end{lemma}
\begin{proof}
We first verify a special case of (a): if $T = \emptyset$, then $T' = \emptyset$. To wit,
if $T = \emptyset$, then for any $x' \in X'$ mapping to $x \in X$, we may pull back a good formal structure for $\calE$ at $x$ to obtain a good formal structure for $\calE'$ at $x'$.

We next verify (b). Apply Proposition~\ref{P:resolution of turning points} to construct a resolution of turning points $g: Y \to X$ of $\calE$. Put $Y' := Y \times_X X'$ and let $g': Y' \to X'$ be the induced morphism; by the previous
paragraph, $g'$ is a resolution of turning points of $\calE'$. Again by pulling back good formal structures,
we see that the irregularity divisor of $\calE$ on $Y$ pulls back to the irregularity divisor of $\calE'$ on $Y'$; this completes the proof of (b).

Finally, we note that (b) implies (a) using Theorem~\ref{T:numerical criterion}.
\end{proof}

\begin{theorem} \label{T:functorial resolution of turning points}
For $\calC$ as in Definition~\ref{D:nabla-module category}, there exist a covariant functor $(Y,W): \calC \to \Sch'$
and a natural transformation $F: (Y,W) \to \iota$ satisfying the following conditions.
\begin{enumerate}
\item[(a)]
For each $(X,Z,\calE) \in \calC$, the morphism $F(X,Z,\calE): Y(X,Z,\calE) \to X$ is a regularizing 
modification of $(X,Z)$.
\item[(b)]
For each $(X,Z, \calE) \in \calC$ such that $(X,Z)$ is regular and the turning locus of $\calE$ on $X$ is empty, $F(X,Z,\calE)$ is an isomorphism.
\item[(c)]
For each morphism $f: (X',Z',\calE') \to (X,Z,\calE)$ in $\calC$, the square
\[
\xymatrix
@C=50pt{
(Y,W)(X',Z',\calE') \ar^{(Y,W)(f)}[r] \ar^{F(X',Z',\calE')}[d] & (Y,W)(X,Z,\calE) \ar^{F(X,Z,\calE)}[d] \\
(X',Z') \ar^{f}[r] & (X,Z)
}
\]
is cartesian in $\Sch'$.
\end{enumerate}
\end{theorem}
\begin{proof}
By part (a) (resp.\ part (b)) of Lemma~\ref{L:regular diff},
blowing up in the turning locus (resp.\ principalization of the irregularity sheaf) is functorial in $\calC$.
We may thus combine either Corollary~\ref{C:resolution procedure1} or Corollary~\ref{C:resolution procedure2}
with functorial nonembedded and embedded resolution of singularities (Theorem~\ref{T:desing1} and Theorem~\ref{T:desing2})
to obtain the desired result.
\end{proof}

\begin{remark}
Let $(X,Z)$ be a schematic pair in which $X$ is a proper variety over $\CC$,
and let $\calE$ be a $\nabla$-module on $\calO_X(*Z)$. 
Adrian Langer (in conjunction with joint work with H\'el\`ene Esnault)
has asked whether one can bound the geometric complexity of a resolution of turning points of $\calE$
in terms of $X$, the rank of $\calE$, and a bound on the irregularity of $\calE$. In dimension 2, one may simply ask whether one can bound the number of point blowups needed to effect a resolution of turning points; in higher dimensions, one can instead ask whether the irregularity sheaf admits a projective resolution of bounded length by vector bundles of bounded rank.

The closely related question is to compute the Chern classes of the underlying bundle of $\calE$ in terms of the irregularity sheaf and other data. Note that in the regular case, the answer also involves the residues of the connection; see \cite{ohtsuki}.

These questions can be thought of as archimedean analogues of some questions of Deligne on the counting of
$\ell$-adic local systems with prescribed ramification on a smooth scheme over a finite field
\cite{deligne-comptage}.
\end{remark}

\subsection{Resolution of turning points on locally ringed spaces}
\label{subsec:resolution other categories}

We next use Theorem~\ref{T:functorial resolution of turning points} to exhibit resolutions of turning points in 
a category of locally ringed spaces satisfying suitable properties. This expands upon the brief discussion given in \cite[\S 8.2]{kedlaya-goodformal2}.

\begin{defn}
A locally ringed space $(X, \calO_X)$ is said to be \emph{over $\QQ$} if the unique morphism $X \to \Spec(\ZZ)$ factors through $\Spec(\QQ)$. In other words, every nonzero integer is invertible on $X$.
\end{defn}

\begin{defn}
Recall that for $X$ a topological space, $\calF$ a sheaf on $X$, and $S$ a subset of $X$, the \emph{stalk} $\calF_{X,S}$ of $\calF$ at $S$ is the direct limit of $\calF(V)$ over all open subsets $V$ of $X$ containing $S$. For $S  = \{x\}$, this agrees with the usual stalk $\calO_{X,x}$. For schemes, this agrees with the usual definition.
\end{defn}

The following definition is not standard, but is convenient here.

\begin{defn} \label{D:excellent lrs}
Let $(X, \calO_X)$ be a locally ringed space. A subset $U$ of $X$ is \emph{excellent} if the following conditions hold.
\begin{enumerate}
\item[(a)] The stalk $\calO_{X,U}$ is an excellent ring. 
\item[(b)]
For each maximal ideal $\frakm$ of $\calO_{X,U}$, 
the preimage of $\frakm$ under the canonical map $U \to \Spec(\calO_{X,U})$ consists of a single point $x$.
\item[(c)]
For all $\frakm, x$ as in (b), the homomorphism $(\calO_{X,U})_\frakm \to \calO_{X,x}$ of local rings induces an isomorphism of maximal-adic completions (and so in particular is a local homomorphism).
\end{enumerate}
We say that $X$ is \emph{excellent} if every point admits a cofinal system of excellent neighborhoods.
(This system need not be closed under pairwise intersections.)
\end{defn}

\begin{lemma} \label{L:excellent regular inclusion}
Let $(X, \calO_X)$  be a locally ringed space.
For any inclusion of excellent subsets $U \subseteq V$,
the morphism $\calO_{X,V} \to \calO_{X,U}$ is regular.
\end{lemma}
\begin{proof}
By Lemma~\ref{L:regular maximal}, it suffices to check that for every maximal ideal $\frakq$ of $\Spec(\calO_{X,U})$
lying over a maximal ideal $\frakp$ of $\Spec(\calO_{X,V})$, the morphism $(\calO_{X,V})_{\frakp} \to (\calO_{X,U})_{\frakq}$ is regular. 
By hypothesis, the preimage of $\frakq$ under the natural map $U \to \Spec(\calO_{X,U})$ consists of a single point 
$x$, and $(\calO_{X,U})_{\frakq} \to \calO_{X,x}$ is a local homomorphism.
Since $(\calO_{X,V})_{\frakp} \to (\calO_{X,U})_{\frakq}$ is also a local homomorphism, so is the composition
$(\calO_{X,V})_{\frakp} \to \calO_{X,x}$; that is, $x$ belongs to the
preimage of $\frakp$ under the natural map $V \to \Spec(\calO_{X,V})$.
It must then be the unique such point, and the homomorphism $(\calO_{X,V})_{\frakp} \to \calO_{X,x}$ must again induce an isomorphism of maximal-adic completions. It follows that $(\calO_{X,V})_{\frakp} \to (\calO_{X,U})_{\frakq}$ itself induces an isomorphism of maximal-adic completions. Since $(\calO_{X,U})_{\frakq}$ is noetherian, the morphism to its completion is faithfully flat; we may therefore apply
Remark~\ref{R:fppf} to deduce that $(\calO_{X,V})_{\frakp} \to (\calO_{X,U})_{\frakq}$ is regular, as desired.
\end{proof}

\begin{defn}
A \emph{differential space} is a triple $(X, \calO_X, \calD_X)$ in which $(X, \calO_X)$ is an excellent locally ringed space over $\QQ$ and $\calD_X$ is a coherent $\calO_X$-module acting on $\calO_X$ via derivations. We say that such a space is \emph{nondegenerate} if for each $x \in X$, $\calO_{X,x}$ is a nondegenerate differential ring.
Note that the blowup of $X$ along a regular center is again a nondegenerate differential ring.

Let $\calI$ be a coherent ideal sheaf on $X$ with nowhere vanishing stalks. A \emph{$\nabla$-module over $\calO_X(*\calI)$} is a vector bundle $\calE$ on $X$ equipped with an action of $\calD_X$ on $\calE(*\calI)$ satisfying the Leibniz rule. 

For $x \in X$, put $X' := \Spec(\calO_{X,x})$, let $Z'$ be the closed subscheme of $X'$ cut out by the stalk at $x$ of the ideal sheaf cutting out $Z$ on $X$, and view the stalk $\calE_x$ as a $\nabla$-module over $\calO_{X'}(*Z')$.
We say that $x$ is a \emph{turning point} for $\calE$ if the closed point of $X'$ is a turning point for $\calE_x$.
Again, we define the \emph{turning locus} of $\calE$ to be the set of turning points.
\end{defn}

In order to discuss resolutions of turning points, we need to consider not individual locally ringed spaces, but entire categories thereof.
\begin{defn}
Let $\calC$ be a category of differential spaces, let $X$ be an object of $\calC$, and let $\calI$ be a coherent ideal sheaf on $X$. A \emph{blowup} of $X$ along $\calI$ is a final object $f: Y \to X$ in the category of $\calC$-objects over $X$ for which the inverse image ideal sheaf $f^{-1}(\calI) \cdot \calO_Y$ is locally principal. Such an object is of course unique up to unique isomorphism if it exists, so we will typically refer to it as ``the'' blowup of $X$ along $\calI$. If the blowup exists for all $X$ and $\calI$, we say that $\calC$ is \emph{closed under blowups}.
\end{defn}

\begin{theorem} \label{T:abstract resolution of turning points}
Let $\calC$ be a category of differential spaces which is closed under blowups.
Let $X$ be an object of $\calC$, let $\calI$ be a coherent ideal sheaf with nowhere vanishing stalks, and let $\calE$ be a $\nabla$-module over $\calO_X(*\calI)$.
Then there exists a morphism $f: Y \to X$ in $\calC$ which is a composition of blowups, such that $f^{-1}(\calI) \cdot \calO$ is locally principal and the turning locus of $f^* \calE$ is empty (that is, $f$ is a \emph{resolution of turning points} of $\calE$ within $\calC$).
Moreover, $f$ can be chosen so that the turning locus equals the complement of the maximal open subspace of $X$ over which $f$ is an isomorphism.
\end{theorem}
\begin{proof}
Since $\calE$ is by definition a locally free $\calO_X$-module, for every $x \in X$ we can find a neighborhood of $X$ on which $\calE$ is finite free. By further shrinking, we can choose such a neighborhood $U$ which is also excellent.
For any such $U$, we may apply Theorem~\ref{T:functorial resolution of turning points} to find a
modification $f: Y \to \Spec(\calO_{X,U})$ which is a composition of blowups,
such that $f$ principalizes the ideal $\calI_U$ and is a resolution of turning points of $\calE_U$. 
Since $\calC$ is closed under blowups, we may emulate these blowups in $\calC$ to achieve the desired result.
\end{proof}

\subsection{Resolution of turning points in geometric categories}

We conclude by constructing resolutions of turning points in some geometric categories other than excellent $\QQ$-schemes. The list of eligible categories is not exhaustive, but should suffice to illustrate the point.

\begin{theorem} \label{T:global}
Theorem~\ref{T:abstract resolution of turning points} applies when $\calC$ is any of the following categories.
\begin{enumerate}
\item[(a)]
The category of smooth schemes over a field $K$ of characteristic $0$.
\item[(b)]
The category of formally smooth formal schemes over a field $K$ of characteristic $0$.
\item[(c)]
The category of rigid analytic spaces over a nonarchimedean field $K$ of characteristic $0$ (with residue field of arbitrary characteristic).
\item[(d)]
The category of Berkovich analytic spaces over a nonarchimedean field $K$ of characteristic $0$.
\item[(e)]
The category of smooth complex-analytic varieties.
\item[(f)]
The category of formally smooth complex-analytic varieties (i.e., formal completions of smooth complex-analytic varieties along closed analytic subspaces).
\end{enumerate} 
\end{theorem}
\begin{proof}
In all cases, the issue is to establish the existence of excellent neighborhoods,
as then Theorem~\ref{T:abstract resolution of turning points} applies.
In cases (a)--(d), the rings corresponding to ``affine building blocks'' of these spaces are excellent:
\begin{enumerate}
\item[(a)]
$K$-algebras of finite type (straightforward);
\item[(b)]
completions of $K$-algebras of finite type (see \cite[Remark~1.2.9]{kedlaya-goodformal2});
\item[(c)]
affinoid algebras over $K$ (see \cite[Satz~3.3]{kiehl});
\item[(d)]
Berkovich affinoid algebras over $K$ (see \cite[Th\'eor\`eme~2.13]{ducros}).
\end{enumerate}
This completely settles cases (a) and (b).

In case (c), we must first comment that in order to get objects in the category of locally ringed spaces (rather than G-locally ringed spaces, defined with respect to a G-topology) we must replace rigid analytic spaces with their associated Huber adic spaces, as in \cite{vanderput-schneider}.
In this context, condition (b) of Definition~\ref{D:excellent lrs} is formal, while 
condition (c) is implied by \cite[Proposition~7.2.2/1]{bgr}. 

In case (d), we again must pass from the original category of Berkovich analytic spaces, which are defined with respect to a G-topology, with their associated reified adic spaces, as in \cite{kedlaya-reified}.
In this context, we may reduce everything to case (c) by passing from $K$ to a sufficiently large extension field.

For (e), we take the neighborhoods in question to be closed polydiscs. 
In Definition~\ref{D:excellent lrs}, condition (a) is covered by \cite[Corollary~3.2.7]{kedlaya-goodformal2},
and condition (c) is straightforward.
To deduce (b), note that if $U$ is a closed polydisc in a complex-analytic manifold $X$
and $I$ is a maximal ideal of $\calO_{X,U}$ which is not the contraction of the maximal ideal at any $x \in X$,
then by compactness there must exist a finite subset $S$ of $I$ whose zero loci have empty intersection in $X$.
There exists some open polydisc $V$ containing $U$ such that the elements of $S$ all extend to $V$
and their zero loci continue to have empty intersection on $V$; since $V$ is a Stein space and the ideal sheaf on $V$ generated by $S$ is trivial, the elements of $S$ generate the unit ideal in $\calO_{X,U}$, contradiction.

For (f), the arguments are similar to those in (c).
\end{proof}

\begin{remark} \label{R:transfer results}
In the same way, we may transfer other results on good formal structures from excellent schemes to the other categories listed in Theorem~\ref{T:global}; this includes the purity theorem (Theorem~\ref{T:purity of the turning locus}) and the construction of the irregularity sheaf (Definition~\ref{D:irregularity sheaf}). In particular,
this process is implicit in our prior assertions that Andr\'e's purity and semicontinuity theorems from \cite{andre},
which are formulated in terms of complex-analytic varieties, can be recovered from Theorem~\ref{T:purity of the turning locus} and Corollary~\ref{C:semicontinuity}.
\end{remark}

\begin{remark}
It may also be possible to go in the other direction, by using algebraization/approximation techniques to transfer resolution of turning points from complex algebraic varieties to excellent schemes. If so, this would mean (roughly) that one could recover all of the results of this paper 
with the primary dependence on \cite{kedlaya-goodformal1, kedlaya-goodformal2}, namely the existence of uncontrolled resolutions of turning points (Proposition~\ref{P:resolution of turning points}), replaced by a dependence on Mochizuki's corresponding result \cite[Theorem~19.5]{mochizuki2} (see also Remark~\ref{R:mochizuki def}).
We have not made a serious attempt to do this.
\end{remark}

\appendix
\section{Errata for \cite{kedlaya-goodformal2}}

We record here an erratum for \cite{kedlaya-goodformal2} pointed out to us by Matthew Morrow.
Therein, the proof of Lemma 3.1.6 is insufficient: while any regular sequence of parameters of $R$ does contain a sequence of parameters of $R_{\mathfrak{q}}$, it need not contain a \emph{regular} sequence of parameters. To give a completed argument, we first observe that the proofs of Lemma 3.1.7, Corollary 3.1.8, and Corollary 3.1.9 do not depend on Lemma 3.1.6, so we may use them freely in what follows.

Let $\partial_1,\dots,\partial_n$ be a sequence of derivations of rational type
with respect to the regular sequence of parameters $x_1,\dots,x_n$ of $R$.
Let $y_1,\dots,y_m$ be a sequence in $R$ which is a regular sequence of parameters of $R_{\mathfrak{q}}$.
Since $\widehat{R}$ satisfies the weak Jacobian criterion by [32, Theorem 100], we may reorder the original 
sequence $x_1,\dots,x_n$ so as to ensure that the $m \times m$ matrix $A$ given by $A_{ij} = \partial_i(y_j)$
has nonzero determinant modulo $\mathfrak{q}$. We may then define the derivations
$\partial'_j = \sum_i (A^{-1})_{ij} \partial_i$ on $R_{\mathfrak{q}}$ for $j=1,\dots,m$.

To complete the proof, we must establish that the derivations $\partial'_1,\dots,\partial'_m$ commute.
To see this, we may assume without loss of generality that $R$ is complete; by Corollary 3.1.8, we then have
$R \cong k\llbracket x_1,\dots,x_n \rrbracket$, so $k \llbracket x_1,\dots,x_m \rrbracket$ is contained in the joint kernel of $\partial'_1,\dots,\partial'_m$. 
By counting dimensions, we see that
$R/\mathfrak{q}$ is finite over $k\llbracket x_1,\dots,x_m \rrbracket$;
we may thus identify the completion of $R_{\mathfrak{q}}$ with $\ell \llbracket y_1,\dots,y_m \rrbracket$ where $\ell$ is the integral closure of the fraction field of 
$k \llbracket x_1,\dots,x_m \rrbracket$ in $R_{\mathfrak{q}}$. 
On this ring, the actions of $\partial'_1,\dots,\partial'_m$ are all $\ell$-linear, so they must coincide with the formal partial derivatives in the variables $y_1,\dots,y_m$; this proves the claim.

In addition, we report one further typo in \cite{kedlaya-goodformal2}:
 in Lemma 3.2.5(a), the reference to [33, Theorem 101] should be to [32, Theorem 101].

\end{document}